\documentclass[12pt,a4paper,reqno]{amsart}

\usepackage{amsmath}

\usepackage{enumitem}
\usepackage[a4paper, margin=1.1in]{geometry}

\usepackage[utf8]{inputenc}
\usepackage[T1]{fontenc}
\usepackage[english]{babel}
\usepackage[babel]{microtype}
\usepackage[dvipsnames]{xcolor}

\usepackage{csquotes}

\usepackage{amsmath,amssymb,amsthm}

\usepackage{mathrsfs}

\usepackage{mathtools}
\usepackage{commath}

\usepackage{thmtools}
\usepackage{thm-restate}

\usepackage{bbm}
\usepackage{dsfont}
\usepackage{lmodern}
\usepackage{fixcmex}

\usepackage{cases}
\usepackage{enumitem}
\setlist[enumerate,1]{label=(\roman*)}

\usepackage{centernot}
\usepackage{booktabs}
\usepackage{multirow}
\usepackage{comment}
\usepackage{float}
\usepackage{scalerel}
\usepackage{xspace}

\usepackage[pdfborderstyle={/S/U/W 0},unicode]{hyperref}
\hypersetup{
  colorlinks=true,
  linkcolor=blue!85!black,
  citecolor=blue!85!black,
  urlcolor=blue!85!black,
}

\usepackage[noabbrev, nameinlink]{cleveref}

\usepackage{etoolbox}

\usepackage{newunicodechar}

\usepackage{subfiles}
\usepackage{pdfpages}

\numberwithin{equation}{section}

\ifdefined\checkunusedreferences
\usepackage{refcheck}

\makeatletter
\newcommand{\alsocheck}[1]{%
  \expandafter\let\csname @@\string#1\endcsname#1%
  \expandafter\DeclareRobustCommand\csname relax\string#1\endcsname[1]{%
  \csname @@\string#1\endcsname{##1}\@for\@temp:=##1\do{\wrtusdrf{\@temp}\wrtusdrf{{\@temp}}}}%
  \expandafter\let\expandafter#1\csname relax\string#1\endcsname
}
\newcommand{\alsocheckrange}[1]{%
  \expandafter\let\csname @@\string#1\endcsname#1%
  \expandafter\DeclareRobustCommand\csname relax\string#1\endcsname[2]{%
  \csname @@\string#1\endcsname{##1}{##2}\wrtusdrf{##1}\wrtusdrf{{##1}}\wrtusdrf{##2}\wrtusdrf{{##2}}}%
  \expandafter\let\expandafter#1\csname relax\string#1\endcsname
}
\makeatother

\alsocheck{\cref}
\alsocheck{\Cref}
\alsocheckrange{\crefrange}
\alsocheckrange{\Crefrange}

\else
\fi

\ifdefined\thmcolor
\declaretheoremstyle[
  shaded={bgcolor=\thmcolor,
  }
]{plain}
\else
\fi

\ifdefined\defcolor
\declaretheoremstyle[
  headfont=\normalfont\bfseries,
  bodyfont=\normalfont,
  shaded={bgcolor=\defcolor}
]{noital}
\else
\declaretheoremstyle[
  headfont=\normalfont\bfseries,
  bodyfont=\normalfont,
]{noital}
\fi

\declaretheorem[style=plain,numberwithin=section,name=Theorem]{theorem}

\declaretheorem[style=plain,sibling=theorem,name=Proposition]{proposition}

\declaretheorem[style=plain,sibling=theorem,name=Lemma]{lemma}

\declaretheorem[style=plain,sibling=theorem,name=Corollary]{corollary}

\declaretheorem[style=plain,sibling=theorem,name=Question]{question}

\declaretheorem[style=plain,numbered=no,name=Theorem]{theoremn-n}
\declaretheorem[style=plain,numbered=no,name=Theorems]{theorems-n}
\declaretheorem[style=plain,numbered=no,name=Proposition]{proposition-n}
\declaretheorem[style=plain,numbered=no,name=Propositions]{propositions-n}
\declaretheorem[style=plain,numbered=no,name=Lemma]{lemma-n}
\declaretheorem[style=plain,numbered=no,name=Lemmas]{lemmas-n}
\declaretheorem[style=plain,numbered=no,name=Corollary]{corollary-n}
\declaretheorem[style=plain,numbered=no,name=Corollaries]{corollaries-n}
\declaretheorem[style=plain,numbered=no,name=Conjecture]{conjecture-n}
\declaretheorem[style=plain,numbered=no,name=Conjectures]{conjectures-n}
\declaretheorem[style=plain,numbered=no,name=Question]{question-n}
\declaretheorem[style=plain,numbered=no,name=Questions]{questions-n}
\declaretheorem[style=plain,numbered=no,name=Claim]{claim-n}
\declaretheorem[style=plain,numbered=no,name=Claims]{claims-n}
\declaretheorem[style=plain,numbered=no,name=Fact]{fact-n}
\declaretheorem[style=plain,numbered=no,name=Facts]{facts-n}
\declaretheorem[style=plain,numbered=no,name=Problem]{problem-n}
\declaretheorem[style=plain,numbered=no,name=Problems]{problems-n}
\declaretheorem[style=plain,numbered=no,name=Open Problem]{openproblem-n}
\declaretheorem[style=plain,numbered=no,name=Open Problems]{openproblems-n}
\declaretheorem[style=plain,numbered=no,name=Challenge]{challenge-n}
\declaretheorem[style=plain,numbered=no,name=Challenges]{challenges-n}
\declaretheorem[style=plain,numbered=no,name=Exercise]{exercise-n}
\declaretheorem[style=plain,numbered=no,name=Exercises]{exercises-n}
\declaretheorem[style=plain,numbered=no,name=Property]{property-n}
\declaretheorem[style=plain,numbered=no,name=Properties]{properties-n}

\declaretheorem[style=noital,numbered=no,name=Remark]{remark-n}
\declaretheorem[style=noital,numbered=no,name=Remarks]{remarks-n}
\declaretheorem[style=noital,numbered=no,name=Definition]{definition-n}
\declaretheorem[style=noital,numbered=no,name=Definitions]{definitions-n}
\declaretheorem[style=noital,numbered=no,name=Construction]{construction-n}
\declaretheorem[style=noital,numbered=no,name=Constructions]{constructions-n}
\declaretheorem[style=noital,numbered=no,name=Observation]{observation-n}
\declaretheorem[style=noital,numbered=no,name=Observations]{observations-n}
\declaretheorem[style=noital,numbered=no,name=Example]{example-n}
\declaretheorem[style=noital,numbered=no,name=Examples]{examples-n}

\newcommand{\defined}{\mathrel{\coloneqq}}

\renewcommand{\le}{\leqslant}
\renewcommand{\leq}{\leqslant}
\renewcommand{\ge}{\geqslant}
\renewcommand{\geq}{\geqslant}

\let\oldexists\exists
\let\exists\relax
\DeclareMathOperator{\exists}{\:\!\oldexists}
\let\oldforall\forall
\let\forall\relax
\DeclareMathOperator{\forall}{\:\!\oldforall}

\newcommand{\st}{\mathbin{\colon}}
\undef{\set}
\DeclarePairedDelimiter{\set}{\lbrace}{\rbrace}

\undef{\emptyset}
\newcommand{\emptyset}{\varnothing}

\DeclarePairedDelimiter{\card}{\lvert}{\rvert}

\undef{\mod}
\newcommand{\mod}[1]{\ (\mathrm{mod}\ #1)}

\undef{\abs}
\DeclarePairedDelimiterX{\abs}[1]
{\lvert}{\rvert}{\ifblank{#1}{\,\cdot\,}{#1}}
\undef{\norm}
\DeclarePairedDelimiterX{\norm}[1]
{\lVert}{\rVert}{\ifblank{#1}{\,\cdot\,}{#1}}
\DeclarePairedDelimiterX{\inner}[2]
{\langle}{\rangle}{\ifblank{#1}{\,\cdot\,}{#1},\ifblank{#2}{\,\cdot\,}{#2}}

\DeclareMathDelimiter{\given}
{\mathbin}{symbols}{"6A}{largesymbols}{"0C}
\DeclareMathOperator{\Prob}{\mathbb{P}}
\DeclarePairedDelimiterXPP{\prob}[1]
{\Prob}{\lparen}{\rparen}{}
{\renewcommand{\given}{\nonscript\;\delimsize\vert\nonscript\;\mathopen{}}#1}
\DeclareMathOperator{\Expec}{\mathbb{E}}
\DeclarePairedDelimiterXPP{\expec}[1]
{\Expec}{\lbrack}{\rbrack}{}
{\renewcommand{\given}{\nonscript\;\delimsize\vert\nonscript\;\mathopen{}}#1}
\DeclareMathOperator{\Var}{Var}
\DeclarePairedDelimiterXPP{\var}[1]
{\Var}{\lparen}{\rparen}{}
{\renewcommand{\given}{\nonscript\;\delimsize\vert\nonscript\;\mathopen{}}#1}
\DeclareMathOperator{\Cov}{Cov}
\DeclarePairedDelimiterXPP{\cov}[2]
{\Cov}{\lparen}{\rparen}{}{#1,#2}

\newcommand{\eps}{\varepsilon}

\newcommand{\NN}{\mathbb{N}}

\newcommand{\RR}{\mathbb{R}}

\newcommand{\ZZ}{\mathbb{Z}}

\newcommand{\dimF}{\dim_\mathrm{F}}
\newcommand{\Span}{\operatorname{span}}
\newcommand{\aff}{\operatorname{aff}}
\newcommand{\dimA}{\dim_{\aff}}

\renewcommand{\subset}{\subseteq}

\newcommand{\citeauthordash}[1]{%
  \begingroup
  \DeclareDelimFormat{multinamedelim}{\textendash}%
  \DeclareDelimFormat{finalnamedelim}{\textendash}%
  \citeauthor{#1}%
  \endgroup
}

\title{Dense sets without large sumsets}

\author{Gabriel Dahia \and Jo\~ao Pedro Marciano \and Victor Souza}

\address{IMPA, Estrada Dona Castorina 110, Jardim Bot\^anico, Rio de Janeiro, 22460-320, Brasil}
\email{\{gabriel.dahia, joao.marciano\}@impa.br}
\address{Department of Mathematics, Cornell University, Ithaca NY 14853, USA; and Sidney-Sussex
College, Cambridge, CB2 3HU, United Kingdom}
\email{vsouza@cornell.edu, vss28@cam.ac.uk}

\usepackage[backend=biber, style=ext-numeric, natbib=true, safeinputenc=true,
  maxbibnames=99, giveninits=true,
url=false, isbn=false, doi=false, articlein=false, mincitenames=99, maxcitenames=100, sorting=nyt]{biblatex}
\DeclareFieldFormat{pages}{#1}

\usepackage{graphicx}

\begin{document}

\begin{abstract}
  We prove, for all fixed $0 < \delta < 1$, and all sufficiently large $n$, that there exists $S
  \subset [n]$ with $\card{S} \ge \delta n$ such that \(A + B \not \subset S\) for all ${A, B
  \subset \NN}$ satisfying
  $$\min\big\{\card{A}, \card{B}\big\} \ge \big(3 + o(1)\big) \frac{\log n }{ \log (1 / \delta)}.$$
  A very recent result of \citeauthor{hernandez2026+growth} shows that our bound is sharp up to a
  factor of 3, and together our results settle a conjecture of \citeauthor{kra2025problems}.
  In fact, we prove that a $\delta$-dense random subset of $[n]$ is a valid choice for $S$ with
  high probability, and that one can take $n^{-\alpha} \le \delta \le 1 - c$ where $c > 0$ is
  fixed and $\alpha > 0$ depends only on the $o(1)$ error, answering another question of the same
  authors in a strong form.
\end{abstract}

\maketitle


\section{Introduction}\label{sec:intro}

A simple consequence of the infinite version of Ramsey's theorem is that for every finite
coloring of $\NN$, there exists a monochromatic sumset
\[
  A + B = \{a + b : a \in A, \, b \in B\}
\]
for infinite sets $A, B \subset \NN$.
Erd\H{o}s conjectured that the same holds in every set $S \subset \NN$ with positive upper
density~\cite{erdos1980old}, that is, every set $S \subset \NN$ satisfying
\[
  \limsup_{n \to \infty} \frac{\card[\big]{S \cap [n]}}{n} > 0,
\]
where $[n] \defined \{1, \ldots, n\}$, should also contain $A + B$ for infinite sets $A, B \subset \NN$.

Erd\H{o}s' conjecture was finally proven by \citet{moreira2019proof} in
\citeyear{moreira2019proof} using ergodic theory.
Later, \citet{kra2024infinite} proved the following strengthening of this result: every $S
\subset \NN$ with positive upper density contains, for any $t \in \NN$, a sumset $A_1 + \cdots +
A_t$, where $A_1, \ldots, A_t \subset \NN$ are infinite.
Motivated by these successes, they formulated a number of questions and conjectures regarding
infinite configurations in sets of natural numbers with positive upper
density~\cite{kra2025problems}, some of which have already been solved~\cite{kra2026density}.

The authors in \cite[Section 4.3]{kra2025problems} also posed various quantitative and finite
analogs of these questions, and this is the focus of our work.
Very recently, \citet{hernandez2026+growth} answered both \cite[Question 4.9]{kra2025problems}
and its coloring version in the negative, establishing that no minimum growth rate exists for
infinite sumset patterns in sets with positive density.
They also made progress on another conjecture of \citet[Conjecture~4.10]{kra2025problems}, which
posits that for every fixed $\delta > 0$, there exist constants $C, c > 0$ such that the
following hold:
\begin{enumerate}
  \item For all $T \subset [n]$ satisfying $\card{T} \ge \delta n$, there are $X, Y \subset \NN$ such that
    \[
      \min\big\{\card{X}, \card{Y}\big\} \ge c \log n \qquad \text{ and } \qquad X + Y \subset T.
    \]\label{item:existenceOfSumsetsInDenseSets}
    \vspace{-10pt}
  \item There exists $S \subset [n]$ with $\card{S} \ge \delta n$ such that for all $A, B \subset
    \NN$ with
    \[
      \min\big\{\card{A}, \card{B}\big\} \ge C \log n,
    \]
    we have $A + B \not \subset S$.
    \label{item:denseSetWithoutLargeSumsets}
\end{enumerate}
\citeauthor{hernandez2026+growth} proved that \ref{item:existenceOfSumsetsInDenseSets} holds with
$c = 1 / \log (1 / \delta)$, and generalized it to $t$-fold sumsets, establishing a finitary analog of
\cite{kra2024infinite}: every $T \subset [n]$ satisfying $\card{T} \ge \delta n$ contains a set
of the form $X_1 + \cdots + X_t$, where
\begin{equation*}
  \min_{i \in [t]} \card{X_i} \ge \left( \frac{\log n}{\log (1 / \delta)} \right)^{1/(t - 1)}.
\end{equation*}
On the other hand, they remarked that \ref{item:denseSetWithoutLargeSumsets} ``seems out of reach''.

Our main contribution is proving that \ref{item:denseSetWithoutLargeSumsets} indeed holds, thus
establishing \cite[Conjecture~4.10]{kra2025problems}.
In fact, we show that we can take $C = \big(3 + o(1)\big) / \log (1 / \delta)$, leaving only a
factor of $3 + o(1)$ between the upper and lower bounds.

\begin{theorem}\label{thm:main}
  For every $0 < \gamma \le 1$ and $c > 0$, there exists $\alpha > 0$ such that the following
  holds for all sufficiently large $n \in \NN$.
  For all $n^{-\alpha} < \delta \le 1 - c$, there is $S \subset [n]$ with $\card{S} \ge \delta n$
  such that for every $A, B \subset \NN$ satisfying
  \[
    \min\big\{\card{A}, \card{B}\big\} \ge \frac{(3 + \gamma) \log n}{\log (1 / \delta)},
  \]
  we have \(A + B \not \subset S\).
\end{theorem}

\citet{kra2025problems} defined $\phi(\delta, n)$ to be the largest integer such that every $S
\subset [n]$ with $\card{S} \ge \delta n$ contains a sumset $A + B$ where $A, B \subset [n]$
satisfy $\min\big\{\card{A}, \card{B}\big\} \ge \phi(\delta, n)$.
Using this notation and combining \Cref{thm:main} with the result of
\citet[Corollary~C]{hernandez2026+growth}, we have the following \namecref{stmt:limInfLimSup},
which settles \cite[Conjecture~4.10]{kra2025problems}, and answers \cite[Question
4.12]{kra2025problems} in the negative.

\begin{corollary}\label{stmt:limInfLimSup}
  For every $0 < \gamma \le 1$ and $c > 0$, there exists $\alpha > 0$ such that
  \begin{equation*}
    \frac{1 - \gamma}{\log (1 / \delta)} \le \frac{\phi(\delta, n)}{\log n} \le \frac{3 +
    \gamma}{\log (1 / \delta)},
  \end{equation*}
  for all $n^{-\alpha} \le \delta \le 1 - c$ and sufficiently large $n$.
\end{corollary}

To prove \Cref{thm:main}, we follow the direction suggested by \citet{kra2025problems}, taking
$S$ to be a $\delta$-random subset of $[n]$, that is, each element of $[n]$ is included in $S$
independently with probability $\delta$.
This leads to the following strengthening of \Cref{thm:main}.

\begin{theorem}\label{stmt:randomMain}
  For every $0 < \gamma \le 1$ and $c > 0$, there exists $\alpha > 0$ such that the following
  holds for all sufficiently large $n \in \NN$.
  For all $n^{-\alpha} < \delta \le 1 - c$, if $S \subset [n]$ is a $\delta$-random set, then
  \begin{equation}\label{eq:probabilityThatWeBound}
    \prob[\Big]{\exists A, B \subset \NN \text{ with } \min\big\{\card{A}, \card{B}\big\} \ge k :
    A + B \subset S} \to 0
  \end{equation}
  as $n \to \infty$, where $k = (3 + \gamma) \log n / \log (1 / \delta)$.
\end{theorem}

Events like the one in \eqref{eq:probabilityThatWeBound} are closely related to the study of
random Cayley sum graphs.
In his pioneering work, \citet{green2005counting} showed that \eqref{eq:probabilityThatWeBound}
holds for $\delta = 1/2$ with ${k = 160 \log n}$ in the case $A = B$, which is the one of interest to
determine the clique number of such graphs.
The leading constant $160$ was then sharpened to its asymptotically optimal value of $2$ by
\citet{green2016counting}, still in the setting of $\delta = 1/2$ and $A = B$.
This was recently extended by \citet{campos2024on} to all
\begin{equation}\label{eq:indNumberCayley}
  (\log n)^{-1/80} \le 1 - \delta \le 1/2
\end{equation}
where \eqref{eq:indNumberCayley} features $1 - \delta$ instead of $\delta$ because they study the
independence number of these graphs.
Later, \citet{nenadov2025+remark} improved the exponent of the $\log$ in
\eqref{eq:indNumberCayley} to $-1/3 + o(1)$, and \citet{alon2025+random} announced another
improvement to $-2 + o(1)$.

In our proof of \Cref{stmt:randomMain}, we build on the techniques of \cite{green2005counting} and
\cite{campos2024on}.
The starting point of our proof is the same as in those works, which was pioneered in
\cite{green2005counting}: first, we decompose the family of all possible sets $A$ and $B$ with
size $k$ into three distinct collections depending on the size of their sumset.
We then take a union bound on the probability in \eqref{eq:probabilityThatWeBound} for each of
the resulting collections, using different techniques for each of them.
For sufficiently large sumsets, we either follow the strategy of \citet{green2005counting} based
on counting equivalence classes of Fre\u{\i}man isomorphisms, adapting his results to the
$A \neq B$ setting, or use a simple argument based on \citeauthor{ruzsa1989application}'s covering lemma.

The most challenging part of our proof is in the range of small sumsets, where we adopt the
general strategy of \citet{campos2024on}.
This allows us to avoid counting sets $A \subset [n]$ with $\card{A} = k$ and $\card{A + A} \le
m$, and count instead fingerprints $F \subset A$ with
\begin{equation}\label{eq:fingerprintProps}
  \card{F} \approx \sqrt{\card{A + A}} \qquad \text{and} \qquad \card{F + F} \approx \card{A + A}.
\end{equation}
To ensure that the second approximation in \eqref{eq:fingerprintProps} holds, the authors of
\cite{campos2024on} prove a ``few-translates'' version of Fre\u{\i}man's lemma in convex
geometry~\cite[Lemma~1.14]{freiman1973foundations}.
By restricting the smaller sumset range to $\card{A + B} \le O(k)$, we can rely on
\Cref{stmt:bltMedium}, a corollary of \citet[Theorem 2]{bollobas2022+large}, to directly
find fingerprints $A' \subset A$ and $B' \subset B$.
We then use the asymmetric version of Fre\u{\i}man's lemma due to \citet{ruzsa1994sum} (see
\Cref{stmt:asymFreiman}) to ensure that $\card{A' + B'}$ is sufficiently large.

\begin{theorem}[{\cite[Theorem 2]{bollobas2022+large}}]\label{stmt:bltMedium}
  For all $\sigma > 0$ and $\eps > 0$, there exists $\beta \ge 1$ such that the following holds.
  For all $A$ and $B$ finite non-empty subsets of an Abelian group, there are subsets $A' \subset
  A^* \subset A$ and $B' \subset B^* \subset B$, with
  \begin{equation}\label{eq:propsOfBltSets}
    \card{A^*} \ge (1 - \eps) \card{A}, \quad \card{B^*} \ge (1 - \eps) \card{B}, \quad
    \card{A'} \le \beta k^{1/2} \quad \text{and} \quad \card{B'} \le \beta k^{1/2},
  \end{equation}
  where $k = \max\big\{\card{A}, \card{B}\big\}$, such that
  \begin{equation}
    \card{A' + B'} \ge \min\big\{ (1 - \eps)\card{A^* + B^*}, \, \sigma \card{A^*}, \, \sigma
    \card{B^*} \big\}.
  \end{equation}
\end{theorem}

With \Cref{stmt:bltMedium}, we proceed similarly to \cite{campos2024on}, using
\citeauthor{chang2002polynomial}'s theorem~\cite{chang2002polynomial} to
reduce the overall number of choices for $A'$ and $B'$ (see \Cref{sec:lowerDoubling}).
In \Cref{sec:midDoubling}, we adapt the arguments of \citet{green2005counting} to the $A \neq B$
setting, and in \Cref{sec:largeDoubling} we describe the simple argument based on
\citeauthor{ruzsa1989application}'s covering lemma to handle very large sumsets.
We combine the results of the preceding sections into a proof of \Cref{stmt:randomMain} in
\Cref{sec:proofOfMain}, where we also deduce \Cref{thm:main}.
The final \namecref{sec:concludingRemarks} is dedicated to concluding remarks.

\section{A random set contains no small sumset}\label{sec:lowerDoubling}

The goal of this \namecref{sec:lowerDoubling} is to prove \Cref{stmt:lowerDoubling}, which will
bound the probability that a pair $A, B \subset [n]$ with $\card{A} = \card{B} = k$ and $\card{A
+ B} \le O(k)$ is contained in a $\delta$-random subset.

\begin{lemma}\label{stmt:lowerDoubling}
  For every $0 < \gamma \le 2$, $c > 0$ and $C > 1$, there exists $\alpha > 0$ such that the
  following holds for all sufficiently large $n \in \NN$.
  For all $n^{-\alpha} < \delta \le 1 - c$, if $S \subset [n]$ is a $\delta$-random set, then
  \begin{equation}\label{eq:probabilityLowerDoubling}
    \prob[\Bigg]{\exists (A, B) \in \binom{[n]}{k}^2,~ \card{A + B} \le C k : A + B \subset S}
    \le \frac{2 k}{n^{\gamma / 64}},
  \end{equation}
  where $k = (2 + \gamma) \log n / \log (1 / \delta)$.
\end{lemma}

Before proving \Cref{stmt:lowerDoubling}, we give some (standard) definitions.
When $X$ and $Y$ are subsets of an Abelian group, we say that a function $\psi : X \to Y$ is a
Fre\u{\i}man homomorphism if for every $a_1, a_2, b_1, b_2 \in X$ such that
\begin{equation*}
  a_1 + a_2 = b_1 + b_2,
\end{equation*}
we have
\begin{equation*}
  \psi(a_1) + \psi(a_2) = \psi(b_1) + \psi(b_2).
\end{equation*}
If such a function $\psi$ is a bijection, and $\psi^{-1}$ is also a Fre\u{\i}man homomorphism, then we
say that $\psi$ is a Fre\u{\i}man isomorphism, and that $X$ and $Y$ are Fre\u{\i}man isomorphic.
Finally, we define $\dimF(X)$, the Fre\u{\i}man dimension of $X$, to be the largest $d \in \NN$ for
which there is a subset of $\ZZ^d$ not contained in any proper affine subspace that is Fre\u{\i}man
isomorphic to $X$.

The first result that we need is the following simple consequence of Fre\u{\i}man's
lemma~\cite[Lemma~1.14]{freiman1973foundations} (see also \cite[Lemma~4.3]{bilu1999structure}).

\begin{proposition}[see {\cite[Lemma~21]{green2005counting}} or
  {\cite[Observation~8.2]{campos2024on}}]\label{stmt:boundOnFreimanDim}
  Let $p \in \NN$ be a prime.
  For all $X \subseteq \ZZ/p\ZZ$, if $\card{X + X} \le \kappa \card{X}$, then $\dimF(X) \le 2
  \kappa - 1$.
\end{proposition}

We also require an appropriate version of \citeauthor{chang2002polynomial}'s
\namecref{stmt:changInZq} (\Cref{stmt:changInZq} below) to show that $X = A \cup B$ is contained
in a small generalized arithmetic progression (GAP), i.e.\ a set $P$ of the form
\[
  P = \bigg\{a_0 + \sum_{i = 1}^d (w_i - 1) a_i :
  (w_1, \ldots, w_d) \in [\ell_1] \times \cdots \times [\ell_d] \bigg\},
\]
where $d = \dim(P)$ is called the dimension of the GAP, the integers $a_1, \ldots, a_d$ are
called the differences, and $\ell_1, \ldots, \ell_d$ are called the side-lengths of the GAP.
A GAP of dimension $d$ is referred to as a $d$-GAP.
When every choice of $(w_1, \ldots, w_d)$ results in a different element of $P$, we further say
that it is proper.

In order to count $d$-GAPs efficiently, we will change the setting from $[n] \subset \ZZ$ to
$\ZZ/p\ZZ$, where $2n \le p \le 4n$ is a prime.
The appropriate version of \citeauthor{chang2002polynomial}'s theorem for $\ZZ/p\ZZ$ is then a
corollary of a result of \citet{cwalina2013linear}, a quantitative improvement of the celebrated
\citeauthordash{green2007freiman} theorem~\cite{green2007freiman}.
We also remark that if $t = \card{X}$ is allowed to be sufficiently large, then
\citet[Lemma~4.1]{nenadov2025+remark} improved the dependencies of the parameters of
\Cref{stmt:changInZq} on $\kappa$.

\begin{theorem}[{\cite[Corollary 8.4]{campos2024on}}]\label{stmt:changInZq}
  There exists $C' > 0$ such that the following holds.
  Let $p \in \NN$ be a prime, and let $\kappa \ge 2$.
  If $X \subseteq \ZZ/p\ZZ$ satisfies
  \begin{equation}\label{eq:requirementsForChangInZq}
    \card{X + X} \le \kappa \card{X} \qquad \text{ and } \qquad C' \kappa^3 (\log \kappa)^2 \leq \card{X}
    \leq \exp(-C' \kappa^4 (\log \kappa)^2) p,
  \end{equation}
  then there is $P \subseteq \ZZ/p\ZZ$, a proper\footnote{The $d$-GAP $P_d \subset \ZZ^d$ obtained
    in \cite{campos2024on} is not necessarily proper, so we apply
    \cite[Theorem~2.1]{green2005notes} to ensure that both it and $P$, the resulting $d$-GAP, are
    proper, increasing their sizes by a factor of at most $4^d d^{6 d^2} = 2^{o(\kappa^4)}$, a
  bound that follows from \Cref{stmt:boundOnFreimanDim} and $d \le \dimF(X)$.} $d$-GAP, such that
  \begin{equation*}
    X \subseteq P, \qquad \card{P} \le \exp(C' \kappa^4 (\log \kappa)^2 ) \card{X}, \qquad \text{ and }
    \qquad d = \dim(P) \le \dimF(X).
  \end{equation*}
\end{theorem}

To obtain a small $d$-GAP for $X = A \cup B$, we will rely on the
\citeauthor{plunnecke1970zahlentheoretische}--\citeauthor{ruzsa1989application}
inequality~\cite{plunnecke1970zahlentheoretische,ruzsa1989application} (see
also \cite[Theorem~7.3.3]{zhao2023graph}) to bound $\card{X + X}$.

\begin{theorem}[\citeauthor{plunnecke1970zahlentheoretische}--\citeauthor{ruzsa1989application} inequality]
  \label{thm:plunnecke}
  Let $\Gamma$ be an Abelian group.
  For all integers $r, s \ge 0$ and finite $A, B \subseteq \Gamma$, if $\card{A + B} \leq \sigma
  \card{A}$, then
  \begin{equation*}
    \card{rB \pm sB} \leq \sigma^{r+s}\card{A}.
  \end{equation*}
\end{theorem}

The final result that we need is the following asymmetric version of Fre\u{\i}man's
lemma~\cite[Lemma~1.14]{freiman1973foundations} due to \citet{ruzsa1994sum}.
We define $\dimA(X)$ for a set $X \subset \RR^d$ to be the dimension
of the smallest affine subspace that contains $X$.

\begin{theorem}[{\cite[Theorem~1, Corollary~1.1]{ruzsa1994sum}}]
  \label{stmt:asymFreiman}
  For finite sets $A, B \subset \RR^d$, if $\card{B} \le \card{A}$ and ${\dimA(A + B) = d}$, then
  \begin{equation*}
    \card{A + B} \geq \card{A} + \sum_{t = 1}^{\card{B} - 1} \min\big\{d, \card{A} - t\big\}.
  \end{equation*}
  In particular,
  \begin{equation*}
    \card{A + B} \geq \card{A} + d \card{B} - \binom{d + 1}{2}.
  \end{equation*}
\end{theorem}

We will use \Cref{stmt:asymFreiman} in the form of the following \namecref{stmt:asymFreimanUnion}.

\begin{corollary}\label{stmt:asymFreimanUnion}
  Let $\Gamma$ be an Abelian group.
  For all $r \ge 1$ and every finite, non-empty $A, B \subset \Gamma$ with $\card{B} \le
  \card{A}$ and ${\dimF(A \cup B) = r}$, we have
  \begin{equation}\label{eq:asymFreimanUnion}
    \card{A + B} \geq \card{A} + \sum_{t = 1}^{\card{B} - 1} \min\big\{r - 1, \card{A} - t\big\}
  \end{equation}
  In particular,
  \begin{equation}\label{eq:freimanLikeBound}
    \card{A + B} \geq \card{A} + (r - 1) \card{B} - \binom{r + 1}{2}.
  \end{equation}
\end{corollary}

\begin{proof}
  Let $X = A \cup B$, $X_1 \subset \ZZ^r$ and $\psi : X \to X_1$ be the Fre\u{\i}man isomorphism that
  witnesses that $\dimF(X) = r$.
  Namely, $X_1$ is not contained in a proper affine subspace.
  Further let $A_1 = \psi(A)$ and $B_1 = \psi(B)$.
  Since $\psi$ is a Fre\u{\i}man isomorphism, we have $\card{A + B} = \card{A_1 + B_1}$.
  Let $d = \dimA(A_1 + B_1)$, fix $a_* \in A_1$ and $b_* \in B_1$, and let $V$ be the
  $d$-dimensional linear subspace
  of $\RR^r$ such that
  \[
    \aff(A_1 + B_1) = a_* + b_* + V,
  \]
  where we denote by $\aff(W)$ the smallest affine subspace that contains $W$.
  Since
  \[
    (a + b_*) - (a_* + b_*) = a - a_* \in V
    \qquad \text{and} \qquad
    (a_* + b) - (a_* + b_*) = b - b_* \in V
  \]
  for all $a \in A_1$ and $b \in B_1$, we have $A_1 \subset a_* + V$ and $B_1 \subset b_* + V$.

  Now, take a linear isomorphism $f : V \to \RR^d$, and set
  \[
    A_2 = f(A_1 - a_*)
    \qquad \text{and} \qquad
    B_2 = f(B_1 - b_*).
  \]
  Thus, $\card{A_2 + B_2} = \card{A_1 + B_1}$ and $\dimA(A_2 + B_2) = d$, so, by
  \Cref{stmt:asymFreiman}, we have
  \begin{equation*}
    \card{A + B} = \card{A_2 + B_2} \ge \card{A} + \sum_{t = 1}^{\card{B} - 1} \min\big\{d,
    \card{A} - t\big\}
  \end{equation*}
  and
  \begin{equation*}
    \card{A + B} = \card{A_2 + B_2} \ge \card{A} + d \card{B} - \binom{d + 1}{2}.
  \end{equation*}
  Since it is clear that $d \le r$, it suffices to show that $r \le d + 1$.

  To prove that $r \le d + 1$, note that
  \[
    X_1 = A_1 \cup B_1 \subset a_* + \Span\big(V \cup \{b_* - a_*\}\big),
  \]
  which is contained in an affine subspace of $\RR^r$ of dimension at most $d + 1$.
  As $X_1$ is not contained in any proper affine subspace of $\RR^r$, we conclude that $r \le d +
  1$, completing the proof.
\end{proof}

We can now prove \Cref{stmt:lowerDoubling}.

\begin{proof}[Proof of \Cref{stmt:lowerDoubling}]
  Let $p$ be a prime with $2n \le p \le 4n$, and fix a map $\psi : {[n] \to \ZZ/p\ZZ}$ which is a
  Fre\u{\i}man isomorphism between $[n]$ and $\psi([n])$.
  Replacing $A, B, S \subset [n]$ respectively with $\psi(A), \psi(B), \psi(S) \subset \ZZ/p\ZZ$ and
  relabeling, the left-hand side of \eqref{eq:probabilityLowerDoubling} is at most
  \begin{equation}\label{eq:probabilityLowerDoublingInZq}
    \prob[\Bigg]{\exists (A, B) \in \binom{\ZZ/p\ZZ}{k}^2,~ \card{A + B} \le C k : A + B \subset S},
  \end{equation}
  so we will now consider $A$ and $B$ to be $k$-sets in $\ZZ/p\ZZ$.

  Fix $\eps = \gamma / 32$ and apply \Cref{stmt:bltMedium} with $\sigma = 2 C$ and this choice of
  $\eps$ to each pair of $k$-sets $A$ and $B$ with $\card{A + B} \le C k$.
  From this and \eqref{eq:propsOfBltSets}, we obtain $A^* \subset A$ and $B^* \subset B$ with
  \begin{equation}\label{eq:sizesInBlt}
    \card{A^*} \ge (1 - \eps) k \qquad \text{and} \qquad \card{B^*} \ge (1 - \eps) k.
  \end{equation}
  Moreover, from that same application of \Cref{stmt:bltMedium}, there exist $A' \subset A^*$ and
  $B' \subset B^*$ with $\card{A'} \le \beta \sqrt{k}$ and
  $\card{B'} \le \beta \sqrt{k}$ satisfying
  \begin{equation}\label{eq:lowerBoundOnSumset}
    \card{A' + B'} \ge \min\big\{ (1 - \eps) \card{A^* + B^*}, 2 C \card{A^*}, 2 C \card{B^*}\big\}.
  \end{equation}
  Note that we trivially have
  \[
    \card{A' + B'} \le \card{A^* + B^*} \le \card{A + B} \le C k < 2 C \min\big\{\card{A^*},
    \card{B^*}\big\}
  \]
  since $\eps < 1/2$, so \eqref{eq:lowerBoundOnSumset} is always at least
  \begin{equation}\label{eq:lowerBoundOnSumsetSimpler}
    \card{A' + B'} \ge (1 - \eps) \card{A^* + B^*}.
  \end{equation}

  Now, let $X = A^* \cup B^*$, and observe that
  \begin{equation}\label{eq:doublingOfUnion}
    \card{X + X} \le \card{A + A} + \card{A + B} + \card{B + B} \le 2 C^2 k + C k \le 3 C^2 k
  \end{equation}
  where the second to last inequality is due to \Cref{thm:plunnecke}, and the last uses $C > 1$.
  Importantly, it follows from \eqref{eq:doublingOfUnion} and \Cref{stmt:boundOnFreimanDim} that
  \begin{equation}\label{eq:boundOnFreimanDimOfUnion}
    r \defined \dimF(X) \le \frac{2 \card{X + X}}{\card{X}} - 1 \le \frac{6 C^2 k}{(1 -
    \eps)k} - 1 < 12 C^2
  \end{equation}
  By \eqref{eq:freimanLikeBound} in \Cref{stmt:asymFreimanUnion}, applied to $A^*$ and $B^*$, we obtain
  \begin{equation}\label{eq:preFreimanLowerBoundOnSumset}
    \card{A^* + B^*} \ge (1 - \eps) r k - \binom{r + 1}{2}
  \end{equation}
  due to \eqref{eq:sizesInBlt}.
  We claim that combining \eqref{eq:preFreimanLowerBoundOnSumset} with
  \eqref{eq:lowerBoundOnSumsetSimpler} yields
  \begin{equation}\label{eq:freimanLowerBoundOnSumset}
    \card{A' + B'} \ge (1 - \eps)^2 r k - \binom{r + 1}{2} \ge (1 - 3 \eps) r k.
  \end{equation}
  Indeed, taking $\alpha$ sufficiently small, we may assume that $k \ge 12 C^2 / \eps
  \ge r / \eps$ by \eqref{eq:boundOnFreimanDimOfUnion}, which implies the last inequality in
  \eqref{eq:freimanLowerBoundOnSumset}.

  Towards an application of \Cref{stmt:changInZq} to $X$, we verify that choosing $\kappa = 6 C^2$
  satisfies \eqref{eq:requirementsForChangInZq}.
  The first inequality of \eqref{eq:requirementsForChangInZq} is satisfied by \eqref{eq:doublingOfUnion}.
  For the lower bound in the second inequality of \eqref{eq:requirementsForChangInZq}, we also
  take $\alpha$ sufficiently small so that
  \begin{equation*}
    \card{X} \ge (1 - \eps) k \ge \frac{k}{2} > C' \kappa^3 (\log \kappa)^2
  \end{equation*}
  as $\eps \le 1/2$.
  It remains to check the upper bound in the second inequality of \eqref{eq:requirementsForChangInZq}.
  Since $\gamma \le 2$ and $\delta \le 1 - c$, we have
  \begin{equation*}
    \card{X} \le \card{A \cup B} \le 2 k \le \frac{8 \log n}{\log(1 / (1 - c))} < \sqrt{n}
    \le \exp(-C' \kappa^4 (\log \kappa)^2) p
  \end{equation*}
  where the last two inequalities use that $n \le p / 2$ is sufficiently large for $C$ and $c$.
  Applying then \Cref{stmt:changInZq} to $X$ with $\kappa = 6 C^2$
  we obtain a $d$-GAP $P$ such that
  \begin{equation}\label{eq:propertiesOfP}
    X \subseteq P,
    \qquad
    \card{P} \le C'' \card{X},
    \qquad \text{and} \qquad
    d \le \dimF(X),
  \end{equation}
  for a suitable constant $C'' > 0$.
  In particular, combining \eqref{eq:propertiesOfP} with \eqref{eq:freimanLowerBoundOnSumset} yields
  \begin{equation}\label{eq:lowerBoundWrtGAP}
    \card{A' + B'} \ge (1 - 3 \eps) r k \ge (1 - 3 \eps) \dim(P) k
  \end{equation}
  where $P$ is the GAP that contains $A^* \supseteq A'$ and $B^* \supseteq B'$.
  As $A' \subset A$, $B' \subset B$ and $A', B' \subset P$,
  \eqref{eq:probabilityLowerDoublingInZq} is at most
  \begin{equation}\label{eq:probabilityLowerDoublingWithGap}
    \sum_{P \in \mathcal{P}} \sum_{t_1 = 1}^{\beta \sqrt{k}} \sum_{t_2 = 1}^{\beta \sqrt{k}}
    \prob[\Bigg]{\exists (A', B') \in \binom{P}{t_1} \times \binom{P}{t_2},~ A' + B' \text{
    satisfies \eqref{eq:lowerBoundWrtGAP}}: A' + B' \subset S},
  \end{equation}
  where $\mathcal{P}$ is a minimal collection of GAPs in $\ZZ/p\ZZ$ that satisfies
  \eqref{eq:propertiesOfP} for some choice of $X = A^* \cup B^*$.

  Taking a union bound over the choices of $(A', B')$, we conclude that
  \eqref{eq:probabilityLowerDoublingWithGap} is at most
  \begin{equation}\label{eq:probabilityAfterUnionBound}
    \sum_{P \in \mathcal{P}} \beta^2 k \binom{\card{P}}{\beta \sqrt{k}}^2 \max_{A', B'}
    \delta^{\card{A' + B'}} \le
    \sum_{d = 1}^{2k} \sum_{P \in \mathcal{P}_d}
    \beta^2 k \binom{2 C'' k}{\beta \sqrt{k}}^2 \delta^{(1 - 3 \eps) d k}
  \end{equation}
  where $\mathcal{P}_d \defined \{P \in \mathcal{P} : \dim(P) = d\}$, and we used
  \eqref{eq:lowerBoundWrtGAP} and the upper bound on the size of $P \in \mathcal{P}$ in
  \eqref{eq:propertiesOfP}.
  For each $d$, we bound the size of $\mathcal{P}_d$ by
  \begin{equation}\label{eq:sizeOfCalP}
    \card{\mathcal{P}_d} \le p^{d + 1} (2 C'' k)^d
    \le (4 n)^{d + 1} (2 C'' k)^d
  \end{equation}
  by choosing the differences and starting point of $P$, and then counting the number of possible
  side-lengths.
  Substituting \eqref{eq:sizeOfCalP} into \eqref{eq:probabilityAfterUnionBound}, we obtain
  \begin{equation*}
    \sum_{d = 1}^{2k} \sum_{P \in \mathcal{P}_d}
    \beta^2 k \binom{2 C'' k}{\beta \sqrt{k}}^2 \delta^{(1 - 3 \eps) d k} \le
    \sum_{d = 1}^{2 k} (4 n)^{d + 1} \, (2 C'' k)^d \, \beta^2 k \, \binom{2 C'' k}{\beta \sqrt{k}}^2
    \delta^{(1 - 3 \eps) d k},
  \end{equation*}
  which, using that $d \le 2k$ and adequately increasing $C''$, is at most
  \begin{equation}\label{eq:finalCalculations}
    \sum_{d = 1}^{2k} \exp\Big((d + 1) \big(\log n + 2 \log(3 C'' k)\big) + 2 \beta \sqrt{k} \log (2
    C'' k) - (1 - 3\eps) d k \log (1/\delta)\Big).
  \end{equation}
  Furthermore, note that the assumption $k = (2 + \gamma) \log n / \log (1 / \delta)$ implies that
  \begin{equation}\label{eq:implicationOfDefK}
    (1 - 3 \eps) d k \log (1 / \delta) = (1 - 3 \eps)(2 + \gamma) d \log n.
  \end{equation}
  Moreover, for every $d \ge 1$, our choice of $\eps = \gamma / 32$ gives
  \[
    (1 - 3\eps)(2 + \gamma)d - (d + 1) - \eps(d + 1)
    \ge \gamma - (8 + 3\gamma)\eps \ge \frac{9\gamma}{16}.
  \]
  It then follows from \eqref{eq:implicationOfDefK} that
  \begin{equation}\label{eq:calculationsWork}
    (1 - 3\eps) d k \log (1/\delta) - (d + 1)\log n \ge \eps (d + 1) \log n \ge 4 \Big( \big(d +
    1 + \beta \sqrt{k} \big) \log(3 C'' k)\Big)
  \end{equation}
  if $n$ is sufficiently large.
  Substituting \eqref{eq:calculationsWork} in \eqref{eq:finalCalculations} and combining the
  preceding estimates, we obtain that \eqref{eq:probabilityLowerDoublingInZq}, and therefore the
  left-hand side of \eqref{eq:probabilityLowerDoubling}, is at most
  \[
    \prob[\Bigg]{\exists (A, B) \in \binom{[n]}{k}^2,~ \card{A + B} \le C k : A + B \subset S} \le
    \sum_{d = 1}^{2 k} n^{-\eps / 2} \le \frac{2k}{n^{\gamma / 64}},
  \]
  as we wanted to show.
\end{proof}


\section{Counting sumsets of moderate size}\label{sec:midDoubling}

The goal of this section is to generalize the counting results of \citet{green2005counting},
which give precise counts of sets $A \subset [n]$ such that $\card{A} = k$ and $\card{A + A} \le
m$, to counting sets of the form $A + B$ for $A, B \subset [n]$ such that $\card{A} = \card{B} =
k$ and $\card{A + B} \le m$.

\begin{theorem}\label{stmt:countingPairsOfSets}
  For all integers $2 \le k \le m \le n$, the number of sets $Y \subset [2n]$
  that can be written as $Y = A + B$ for $A, B \subset [n]$ satisfying
  \begin{equation}\label{eq:propsOfPairs}
    \card{A} = \card{B} = k \qquad \text{and} \qquad \card{A + B} \le m < k (k + 1) / 2
  \end{equation}
  is at most
  \begin{equation}\label{eq:numberOfPairs}
    n^{2 m / k + 1} \, k^{24 k}.
  \end{equation}
  Moreover, for every $0 < \xi \le 2^{-8}$, if
  \begin{equation}\label{eq:additionalAssumptions}
    k \ge 2^{200} \qquad \text{and} \qquad m \le k^{1 + \xi},
  \end{equation}
  the number of such sets is at most
  \begin{equation}\label{eq:numberOfPairsWithSmallDoubling}
    n^{2 m / k + 1} \, (6 e m / k)^{4 k} \, \exp\big(k^{1 - 2\xi}\big).
  \end{equation}
\end{theorem}

To prove \Cref{stmt:countingPairsOfSets}, we will rely on \Cref{stmt:asymFreimanUnion}, the
standard \citeauthor{plunnecke1970zahlentheoretische}--\citeauthor{ruzsa1989application}
inequality (\Cref{thm:plunnecke}) and the following integer variant of
\cite[Proposition~6.1]{green2016counting}.
Its proof follows directly from the same argument as in \cite[Proposition~23]{green2005counting},
but in the simpler case of $[n] \subset \ZZ$ rather than in $\ZZ/n\ZZ$, for $n$ being a
sufficiently large prime.

\begin{lemma}[\cite{green2005counting,green2016counting}]\label{stmt:countingUnions}
  Let $n, r, s, t \in \NN$.
  The number of sets $X \subset [n]$ such that $\card{X} = t$ and $\dimF(X) \le r$ is at most
  \begin{equation}\label{eq:numberOfSetsWithGivenDim}
    n^{r + 1} t^{4 t}.
  \end{equation}
  Moreover, if $\card{X + X} \le s \le t^{31 / 30} / 2$, then there are at most
  \begin{equation}\label{eq:numberOfSetsWithGivenDimAndSmallDoubling}
    n^{r + 1} \left(\frac{2 e s}{t}\right)^t \exp(t^{31/32})
  \end{equation}
  such sets $X$.
\end{lemma}

We are now ready to prove \Cref{stmt:countingPairsOfSets}.

\begin{proof}[Proof of \Cref{stmt:countingPairsOfSets}]
  Observe that, for each $Y \subset [2n]$ such that there exist $A, B \subset [n]$ satisfying
  \eqref{eq:propsOfPairs} and $Y = A + B$, we have that $X = A \cup B$ satisfies $k \le \card{X} \le 2k$.
  Trivially, then, given such $X$, there exist $A, B \subset X$ with $\card{A} =
  \card{B} = k$ such that $A + B = Y$.
  If we moreover assume that \eqref{eq:propsOfPairs} holds, then by the
  \citeauthor{plunnecke1970zahlentheoretische}--\citeauthor{ruzsa1989application} inequality,
  \Cref{thm:plunnecke}, we have
  \begin{equation*}
    \card{A + A} \leq \frac{m^2}{k} \qquad \text{and} \qquad
    \card{B + B} \leq \frac{m^2}{k},
  \end{equation*}
  and therefore $X$ also satisfies
  \begin{equation}\label{eq:doublingOfX}
    \card{X + X} \le \card{A + A} + \card{A + B} + \card{B + B}
    \le \frac{3 m^2}{k}.
  \end{equation}

  To count the sets $Y \subset [2n]$ that can be written as $Y = A + B$ for $A, B \subset [n]$
  satisfying \eqref{eq:propsOfPairs}, we first count the associated $X$ using
  \eqref{eq:numberOfSetsWithGivenDim} in \Cref{stmt:countingUnions}, and then choose $A, B
  \subset X$.
  This gives at most
  \begin{equation}\label{eq:almostNumberOfSetsWithGivenDim}
    \sum_{t = k}^{2k} n^{r + 1} \, t^{4 t} \, \binom{t}{k}^2
    \le 2k \, n^{r + 1} \, (2 k)^{4 (2 k)} \, \binom{2k}{k}^2
    \le n^{r + 1} \, k^{18 k} \, 8^{2 k}
    \le n^{r + 1} \, k^{24 k}
  \end{equation}
  choices for $Y$, where $r \ge \dimF(X)$ is an upper bound on the Fre\u{\i}man dimension of all
  such $X$, and we used $k \ge 2$ in the last inequality.
  Under the assumptions in \eqref{eq:additionalAssumptions}, we apply
  \eqref{eq:numberOfSetsWithGivenDimAndSmallDoubling} in \Cref{stmt:countingUnions} instead of
  \eqref{eq:numberOfSetsWithGivenDim},
  taking $t = \card{X} \ge k$ and $s = 3 m^2 / k$.
  This choice is valid since $s \le 3 k^{1 + 2\xi}$ by \eqref{eq:doublingOfX}, and $k^{1/ 30 - 2\xi}
  \ge 6$ implies $s \le t^{31 / 30} / 2$.
  As a consequence, we obtain
  \begin{equation}\label{eq:almostNumberOfSetsWithGivenDimAndSmallDoubling}
    \sum_{t = k}^{2k} n^{r + 1} \, \left(\frac{2 e s}{t}\right)^t \, \exp(t^{31/32}) \, \binom{t}{k}^2
    \le n^{r + 1} \, \left(\frac{6 e m}{k}\right)^{4 k} \, \exp\big(k^{1 - 2\xi}\big)
  \end{equation}
  where the inequality follows because
  \[
    2k \exp\big((2k)^{31 / 32}\big) \le \exp(k^{1 - 2\xi})
  \]
  due to $\xi \le 2^{-8}$ and $k \ge 2^{200}$, and
  \[
    \left(\frac{2 e s}{k}\right)^{2k} \binom{2k}{k}^2 \le \left(\frac{6 e m}{k}\right)^{4 k}
  \]
  again by $m \ge k$ and our choice of $s = 3 m^2 / k$.

  To complete the proof, note that \eqref{eq:numberOfPairs} and
  \eqref{eq:numberOfPairsWithSmallDoubling} would respectively match
  \eqref{eq:almostNumberOfSetsWithGivenDim} and
  \eqref{eq:almostNumberOfSetsWithGivenDimAndSmallDoubling} if we can bound $\dimF(X) = r \le 2 m / k$.
  Towards that goal, observe that if $r \ge k$, then
  \Cref{stmt:asymFreimanUnion}\eqref{eq:asymFreimanUnion} would imply that
  \[
    \card{A + B} \ge k + \sum_{t = 1}^{k - 1} \min\{k - 1, k - t\} \ge \frac{(k + 1) k}{2},
  \]
  so it follows from the assumption $m < (k + 1) k / 2$ in \eqref{eq:propsOfPairs} that $r \le k - 1$.
  With that assumption, we can apply \Cref{stmt:asymFreimanUnion}\eqref{eq:freimanLikeBound} to
  conclude that
  \[
    m \ge \card{A + B} \ge r k - \binom{r + 1}{2} \ge r \left(k - \frac{(k - 1) + 1}{2}\right)
    \ge \frac{r k}{2},
  \]
  and equivalently that $r \le 2 m / k$, which completes the proof.
\end{proof}

We complete this \namecref{sec:midDoubling} with the following \namecref{stmt:midDoubling} of
\Cref{stmt:countingPairsOfSets}, which we will use in the proof of \Cref{stmt:randomMain}.

\begin{corollary}\label{stmt:midDoubling}
  For every $0 < \gamma \le 2$ and $c > 0$, there exist $C > 1$ and $\alpha > 0$ such that
  the following holds for all sufficiently large $n \in \NN$.
  For all $n^{-\alpha} < \delta \le 1 - c$, if $S \subset [n]$ is a $\delta$-random set, then
  \begin{equation}\label{eq:probabilityMidDoubling}
    \prob[\Bigg]{\exists (A, B) \in \binom{[n]}{k}^2,~ C k < \card{A + B} < \frac{(k + 1)k}{2} : A + B
    \subset S} \le \frac{2 k^2}{n^2},
  \end{equation}
  where $k = (2 + \gamma) \log n / \log (1 / \delta)$.
\end{corollary}

\begin{proof}
  Fix $\xi = 2^{-8}$ and choose $C = C(\gamma, c)$ sufficiently large so that
  \begin{equation}\label{eq:choiceOfC}
    \gamma C \ge 4, \qquad \frac{48 \log C}{\log(1 / (1 - c))} + 1 \le \frac{\gamma C}{2}
    \qquad \text{and} \qquad
    4 \log(6 e C) + 1 \le 12 \log C.
  \end{equation}
  Taking a union bound over the possible $Y = A + B$, separating in two terms depending on
  whether $\card{Y} \le k^{1 + \xi}$, choosing $\alpha \le 2^{-200}$ so that $k \ge 2^{200}$ and
  applying \Cref{stmt:countingPairsOfSets}, we obtain that the left-hand side of
  \eqref{eq:probabilityMidDoubling} is at most
  \begin{equation}\label{eq:midRangeAfterUnionBound}
    \sum_{m = C k}^{k^{1 + \xi}} n^{2m / k + 1} \, (6 e m / k)^{4 k} \, \exp\big(k^{1 - 2\xi}\big) \,
    \delta^m
    + \sum_{m = k^{1 + \xi}}^{(k + 1)k/2} n^{2m / k + 1} \, k^{24 k} \, \delta^m.
  \end{equation}
  Now, for $m \ge C k$, it follows from $k = (2 + \gamma) \log n / \log (1 / \delta)$,
  \eqref{eq:choiceOfC} and monotonicity that
  \begin{equation}\label{eq:deltaPowerM}
    \delta^m = \exp\hspace{-3pt}\left(\hspace{-3pt}- (2 + \gamma) \frac{m}{k} \log n\hspace{-2pt}\right)
    \quad \text{and} \quad
    \left(\frac{6 e m}{k}\right)^{\hspace{-3pt}4 k} \exp\hspace{-3pt}\big(k^{1 - 2\xi}\big) \le
    \exp\hspace{-3pt}\left(\hspace{-2pt}12 k \log \frac{m}{k}\hspace{-2pt}\right).
  \end{equation}
  Using \eqref{eq:deltaPowerM}, we see that \eqref{eq:midRangeAfterUnionBound} is at most
  \begin{equation}\label{eq:midRangeAfterDeltaSubstitution}
    \begin{aligned}
      \sum_{m = C k}^{k^{1 + \xi}} \exp\left(\!\left(\frac{2m}{k} + 1\!\right)\!\log n + 12 k
        \log \frac{m}{k}
      - (2 + \gamma) \frac{m}{k} \log n\right) \\
      + \sum_{m = k^{1 + \xi}}^{(k + 1)k/2} \exp\left(\!\left(\frac{2 m}{k} + 1\!\right)\!\log n + 24 k
      \log k - (2 + \gamma) \frac{m}{k} \log n\right).
    \end{aligned}
  \end{equation}
  For the first sum in \eqref{eq:midRangeAfterDeltaSubstitution}, since $\delta \le 1 - c$, we have
  \begin{equation}\label{eq:upperBoundOnK}
    k \le \frac{4 \log n}{\log\!\big(1 / (1 - c)\big)}.
  \end{equation}
  It follows from \eqref{eq:choiceOfC} and monotonicity that
  \[
    \gamma \frac{m}{k} \log n - 12 k \log \frac{m}{k} - \log n
    \ge \gamma \frac{m}{2 k} \log n,
  \]
  for every $m \ge C k$.
  Therefore,
  \begin{equation}\label{eq:firstEstimate}
    \sum_{m = C k}^{k^{1 + \xi}} \exp\left(\!\left(\frac{2m}{k} + 1\!\right)\!\log n + 12 k \log \frac{m}{k}
    - (2 + \gamma) \frac{m}{k} \log n\right) \le
    \frac{k^2}{n^{\gamma C / 2}} \le \frac{k^2}{n^2}.
  \end{equation}
  Take $\alpha$ sufficiently small so that $k^\xi \ge 4 / \gamma$.
  The upper bound on $k$, \eqref{eq:upperBoundOnK}, also implies that, for all sufficiently large $n$,
  \[
    24 k \log k \le \frac{\gamma}{4} k^\xi \log n.
  \]
  Since our choice of $\alpha$ also gives $\log n \le (\gamma / 4) k^\xi \log n$, it follows that,
  for $m/k \ge k^\xi$,
  \[
    \gamma \frac{m}{k} \log n - 24 k \log k - \log n
    \ge \frac{\gamma}{2} k^\xi \log n \ge 2 \log n,
  \]
  and thus
  \begin{equation}\label{eq:secondEstimate}
    \sum_{m = k^{1 + \xi}}^{(k + 1)k/2} \exp\left(\!\left(\frac{2 m}{k} + 1\right)\!\log n + 24 k
    \log k - (2 + \gamma) \frac{m}{k} \log n\right) \le \frac{k^2}{n^2}.
  \end{equation}
  Combining \eqref{eq:firstEstimate} and \eqref{eq:secondEstimate} establishes
  \eqref{eq:probabilityMidDoubling}.
\end{proof}

\section{A simple bound effective for very large sumsets}\label{sec:largeDoubling}

The goal of this \namecref{sec:largeDoubling} is to prove the following estimate on the
probability of a $\delta$-random subset of $[n]$ containing a large $A + B$ where $A, B \subset [n]$.

\begin{lemma}\label{stmt:largeDoubling}
  For every $0 < \gamma \le 1$ and $c > 0$, there exists $\alpha > 0$ such that the following
  holds for all sufficiently large $n \in \NN$.
  For all $n^{-\alpha} < \delta \le 1 - c$, if $S \subset [n]$ is a $\delta$-random set, then
  \begin{equation}\label{eq:probabilityLargeDoubling}
    \prob[\Bigg]{\exists (A, B) \in \binom{[n]}{k}^2,~  \card{A + B} \ge \frac{(k + 1)k}{2} : A + B
    \subset S} \le \frac{k^2}{n^2},
  \end{equation}
  where $k = (3 + \gamma) \log n / \log (1 / \delta)$.
\end{lemma}

Its trivial proof is a direct consequence of \citeauthor{ruzsa1999analog}'s covering lemma.

\begin{lemma}[\cite{ruzsa1999analog}, see {\cite[Theorem 7.4.1]{zhao2023graph}}]\label{stmt:ruzsaCovering}
  Let $A$ and $B$ be finite nonempty subsets of an additive group $G$.
  If \(\card{A+B} \le \kappa \card{B},\) then there exists a subset $Z \subseteq A$ such that
  \[
    A \subseteq Z + B - B
  \]
  and \(\card{Z} \le \kappa\).
\end{lemma}

We will use \Cref{stmt:ruzsaCovering} in the form of the following \namecref{stmt:trivialCountOfSets}.

\begin{corollary}\label{stmt:trivialCountOfSets}
  For all $n, k, m \in \NN$, there are at most
  \begin{equation}\label{eq:trivialCountOfSets}
    n^{k + m/k} \, \binom{m k}{k}
  \end{equation}
  sets $A, B \subset [n]$ such that $\card{A} = \card{B} = k$ and ${\card{A + B} \le m}$.
\end{corollary}

\begin{proof}
  It follows from \Cref{stmt:ruzsaCovering} that, for fixed $B \subset [n]$ and every $A \subset
  [n]$ such that $\card{A} = \card{B} = k$ and $\card{A + B} \le m$, there exists $Z \subset A$
  with $A \subset Z + B - B$ and $\card{Z} = m / k$, by including more elements in $Z$ if necessary.
  So we choose $B$ with $\binom{n}{k}$ choices, then choose $Z$ with $n^{m/k}$ choices, and
  finally $A \subset Z + B - B$ with
  \[
    \binom{\card{Z + B - B}}{\card{A}} \le \binom{\card{Z} \card{B}^2}{k} \le \binom{m k}{k}
  \]
  choices, all of which combined establish \eqref{eq:trivialCountOfSets}.
\end{proof}

We can now prove \Cref{stmt:largeDoubling}.

\begin{proof}[Proof of \Cref{stmt:largeDoubling}]
  Take a union bound over all sets $A, B \subset [n]$ such that \[
    \card{A} = \card{B} = k
    \qquad \text{and} \qquad
    \card{A + B} \ge \frac{k (k + 1)}{2}
  \]
  in \eqref{eq:probabilityLargeDoubling} to obtain, by
  \eqref{eq:trivialCountOfSets},
  \begin{equation*}
    \prob[\Bigg]{\exists (A, B) \in \binom{[n]}{k}^2,~  \card{A + B} \ge \frac{(k + 1)k}{2} : A + B
    \subset S} \le \sum_{m = m_0}^{k^2} n^{k + m / k} \binom{m k}{k} \delta^m,
  \end{equation*}
  where we set $m_0 = k(k + 1) / 2$.
  Now, since
  \begin{equation*}
    \delta^m = \exp\left(-(3 + \gamma) \frac{m}{k} \log n\right),
  \end{equation*}
  follows from $k = (3 + \gamma) \log n / \log (1 / \delta)$, we obtain
  \begin{equation*}
    \sum_{m = m_0}^{k^2} n^{k + m / k} \binom{m k}{k} \delta^m \le \sum_{m = m_0}^{k^2}
    \exp\left(\frac{3 m}{k} \log n + 3 k \log k - (3 + \gamma)
    \frac{m}{k} \log n\right),
  \end{equation*}
  where we also bounded $k \le 2 m / k$ from the fact that $m \ge (k + 1)k / 2$ in the sum.
  The proof is complete once we note that, for sufficiently large $n$, we have $\gamma > 18 \log
  k / \log n$, since $\gamma$ is fixed and $k \le 4 \log n / \log(1 / (1 - c))$.
  Therefore, we have
  \begin{equation*}
    \sum_{m = m_0}^{k^2} \exp\left(\frac{3 m}{k} \log n + 3 k \log k - (3 +
    \gamma) \frac{m}{k} \log n\right)
    \le \sum_{m = m_0}^{k^2} n^{- \gamma k / 3}
    \le \frac{k^2}{n^2},
  \end{equation*}
  where the last inequality follows from $k \ge 3 / \alpha$ if we choose $\alpha \le \gamma / 2$.
\end{proof}

\section{\texorpdfstring{Proof of \Cref{stmt:randomMain} and \Cref{thm:main}}{Proof of Theorem
1.3 and Theorem 1.1}}\label{sec:proofOfMain}

The proof of \Cref{stmt:randomMain} is a straightforward combination of the previous results.

\begin{proof}[Proof of \Cref{stmt:randomMain}]
  Fix $\gamma' = 1 + \gamma$, and let $C > 1$ and $\alpha_2 > 0$ be given by
  \Cref{stmt:midDoubling} with parameters $\gamma'$ and $c$.
  Let $\alpha_1$ and $\alpha_3$ be given by \Cref{stmt:lowerDoubling} with parameters $\gamma'$,
  $c$ and $C$, and by \Cref{stmt:largeDoubling} with parameters $\gamma$ and $c$, respectively,
  and set $\alpha = \min\{1/2, \alpha_1, \alpha_2, \alpha_3\}$.
  Partition the pairs of sets $A, B \subset [n]$ with $\card{A} = \card{B} = k$ into three
  collections $\mathcal{S}_1$, $\mathcal{S}_2$, and $\mathcal{S}_3$, defined by
  \begin{equation*}
    \begin{aligned}
      \mathcal{S}_1 &= \left\{(A, B) \in \binom{[n]}{k}^2,~ \card{A + B} \le C k \right\}, \\
      \mathcal{S}_2 &= \left\{(A, B) \in \binom{[n]}{k}^2,~ C k < \card{A + B} < \frac{(k + 1)k}{2}
      \right\}, \quad \text{and} \\
      \mathcal{S}_3 &= \left\{(A, B) \in \binom{[n]}{k}^2,~ \frac{(k + 1)k}{2} \le \card{A + B} \right\}.
    \end{aligned}
  \end{equation*}

  Let $S \subset [n]$ be a $\delta$-random set, where $n^{-\alpha} < \delta \le 1 - c$, and observe
  that if there are $A, B \subset \NN$ with $\min\{\card{A}, \card{B}\} \ge k$ and $A + B \subset
  S$, then taking $k$-element subsets of $A$ and $B$ yields a pair in $\mathcal{S}_1 \cup
  \mathcal{S}_2 \cup \mathcal{S}_3$.
  We then bound
  \begin{equation*}
    \prob[\Big]{\exists A, B \subseteq \NN \st \hspace{-2pt} \min\big\{\card{A},
    \card{B}\big\} \ge k,  A + B \subset S} \hspace{-2pt} \le \hspace{-2pt} \sum_{i = 1}^3
    \prob[\big]{\exists (A, B) \in \mathcal{S}_i : A + B \subset S},
  \end{equation*}
  by applying \Cref{stmt:lowerDoubling}, \Cref{stmt:midDoubling}, and \Cref{stmt:largeDoubling}
  with the above parameters, obtaining as a result
  \begin{equation*}
    \prob[\Big]{\exists A, B \subseteq \NN \text{ with } \min\big\{\card{A}, \card{B}\big\} \ge k
    : A + B \subset S} \le \frac{2k}{n^{\gamma' / 64}} + \frac{3k^2}{n^2}
    \le \frac{3k^2}{n^{1 / 64}},
  \end{equation*}
  where the last inequality follows from $\gamma' \ge 1$ and $n$ being sufficiently large,
  completing the proof.
\end{proof}

\Cref{thm:main} now follows easily.

\begin{proof}[Proof of \Cref{thm:main}]
  Fix $0 < \gamma \le 1$ and $c > 0$, and set $\gamma' = \gamma / 2$.
  Take $\alpha > 0$ sufficiently small that $\alpha \le 1/2$ and \Cref{stmt:randomMain} holds with
  parameters $\gamma'$ and $c / 2$.
  For $n^{-\alpha} < \delta \le 1 - c$, choose $\delta'$ with $\delta < \delta' \le 1 - c / 2$
  and $\delta' / \delta \ge 1 + 2^{-r}$ for some fixed $r > 0$ depending only on $\gamma$ and
  $c$, satisfying
  \[
    \frac{3 + \gamma'}{\log(1 / \delta')} \le \frac{3 + \gamma}{\log(1 / \delta)}.
  \]
  Let
  \[
    k = (3 + \gamma) \frac{\log n}{\log(1 / \delta)}
    \qquad \text{and} \qquad
    k' = (3 + \gamma') \frac{\log n}{\log(1 / \delta')},
  \]
  so that $n^{-\alpha} < \delta' \le 1 - c / 2$ and $k' \le k$.

  We show that, with high probability, a $\delta'$-random subset $S$ of $[n]$ satisfies
  $\card{S} \ge \delta n$ and every $A, B \subset \NN$ with
  \[
    \min\big\{\card{A}, \card{B}\big\} \ge k = (3 + \gamma) \frac{\log n}{\log(1 / \delta)}
  \]
  satisfies $A + B \not \subset S$.
  First, since $\card{S}$ is a binomial random variable with mean $\delta' n$, the Chernoff bound gives
  \[
    \prob[\big]{\card{S} < \delta n} \le \exp\big(-2^{-2r - 3}\delta n\big) = o(1),
  \]
  where we used $\alpha \le 1/2$, so it suffices to show that
  \[
    \prob[\Big]{\exists A, B \subseteq \NN \text{ with } \min\big\{\card{A}, \card{B}\big\} \ge k
    : A + B \subset S} = o(1).
  \]
  By \Cref{stmt:randomMain} applied with $\gamma'$, $c / 2$ and $\delta'$, and the fact that
  $k' \le k$, the last probability is at most
  \[
    \frac{3(k')^2}{n^{1 / 64}} = o(1),
  \]
  completing the proof.
\end{proof}

\section{Concluding remarks}\label{sec:concludingRemarks}

In this work, we proved \cite[Conjecture~4.10]{kra2025problems} and answered
\cite[Question~4.12]{kra2025problems} in the negative.
Together with the settling of \cite[Question~4.9]{kra2025problems} by
\citet{hernandez2026+growth}, this leaves a single open problem in the ``Quantitative versions''
section in the work of \citet{kra2025problems}, which we reproduce below using the notation
introduced in \Cref{sec:intro}.

\begin{question}[{\cite[Question 4.11]{kra2025problems}}]\label{stmt:lim}
  Does the limit
  \begin{equation}\label{eq:lim}
    \lim_{n \to \infty} \frac{\phi(\delta, n)}{\log n}
  \end{equation}
  exist, and if so, what is it?
\end{question}

In spite of \Cref{stmt:limInfLimSup} showing that there is only a factor of 3 in the way of
proving that \eqref{eq:lim} exists, we believe that our methods cannot resolve the
\namecref{stmt:lim} completely.
To see that, observe that taking $A = B$ and $\delta = 1/2$ in \Cref{stmt:randomMain}
leads\footnote{Ignoring the replacement of $[n] \subset \ZZ$ by $\ZZ/n\ZZ$.} to a bound on the
clique number of the random Cayley sum graph, a parameter for which we know a lower bound of $2 -
o(1)$~\cite{green2016counting,campos2024on}.
The approach of taking $S$ to be a $\delta$-uniform random subset of $[n]$ therefore cannot
improve the upper bound of $\phi(\delta, n)$ to anything better than $2 \log n / \log
(1/\delta)$, asymptotically.
Note that we could not prove $2 + o(1)$ for the upper bound only because the techniques of
\citet{green2016counting} for the large sumset regime do not transfer to the $A \neq B$ setting,
and we had instead to rely on the estimates in \Cref{sec:largeDoubling}.

\section*{Acknowledgments}

We would like to thank Rob Morris for reading a previous draft of this work, and for suggesting
improvements and pointing out corrections.
We would also like to thank Marcelo Campos for useful discussions.

\renewcommand*{\bibfont}{\normalfont\small}
\printbibliography

\end{document}